\documentclass[12pt]{amsart}
\usepackage{txfonts}      
\usepackage{amssymb}
\usepackage{eucal}
\usepackage{amsmath}
\usepackage{amscd}
\usepackage{xcolor}
\usepackage{multicol}
\usepackage[all]{xy}           
\usepackage{graphicx}
\usepackage{color}
\usepackage{colordvi}
\usepackage{xspace}
\usepackage{tikz}
\usepackage{makecell}
\usepackage{appendix}
\usepackage{amsthm}

\usepackage{ifpdf}
\ifpdf
\usepackage[colorlinks,final,backref=page,hyperindex]{hyperref}
\else
\usepackage[colorlinks,final,backref=page,hyperindex,hypertex]{hyperref}
\fi

\usepackage[active]{srcltx} 

\begin{document}


\newtheorem{thm}{Theorem}[section]
\newtheorem{lem}[thm]{Lemma}
\newtheorem{cor}[thm]{Corollary}
\newtheorem{pro}[thm]{Proposition}
\theoremstyle{definition}
\newtheorem{defi}[thm]{Definition}
\newtheorem{ex}[thm]{Example}
\newtheorem{rmk}[thm]{Remark}
\newtheorem{pdef}[thm]{Proposition-Definition}
\newtheorem{condition}[thm]{Condition}

\renewcommand{\labelenumi}{{\rm(\alph{enumi})}}
\renewcommand{\theenumi}{\alph{enumi}}

\newcommand {\emptycomment}[1]{} 

\newcommand{\nc}{\newcommand}
\newcommand{\delete}[1]{}

\nc{\tred}[1]{\textcolor{red}{#1}}
\nc{\tblue}[1]{\textcolor{blue}{#1}}
\nc{\tgreen}[1]{\textcolor{green}{#1}}
\nc{\tpurple}[1]{\textcolor{purple}{#1}}
\nc{\tgray}[1]{\textcolor{gray}{#1}}
\nc{\torg}[1]{\textcolor{orange}{#1}}
\nc{\tmag}[1]{\textcolor{magenta}}
\nc{\btred}[1]{\textcolor{red}{\bf #1}}
\nc{\btblue}[1]{\textcolor{blue}{\bf #1}}
\nc{\btgreen}[1]{\textcolor{green}{\bf #1}}
\nc{\btpurple}[1]{\textcolor{purple}{\bf #1}}

\nc{\revise}[1]{\textcolor{blue}{#1}}


\nc{\tforall}{\ \ \text{for all }}
\nc{\hatot}{\,\widehat{\otimes} \,}
\nc{\complete}{completed\xspace}
\nc{\wdhat}[1]{\widehat{#1}}

\nc{\ts}{\mathfrak{p}}
\nc{\mts}{c_{(i)}\ot d_{(j)}}

\nc{\NA}{{\bf NA}}
\nc{\LA}{{\bf Lie}}
\nc{\CLA}{{\bf CLA}}

\nc{\cybe}{CYBE\xspace}
\nc{\nybe}{NYBE\xspace}
\nc{\ccybe}{CCYBE\xspace}

\nc{\preliecom}{pre-Lie commutative\xspace}
\nc{\transpreliecom}{transposed pre-Lie commutative\xspace}
\nc{\transNovikovpoisson}{transposed Novikov-Poisson\xspace}
\nc{\transdiffnovikovpoisson}{transposed differential Novikov-poisson\xspace}
\nc{\diffnovikovpoisson}{differential Novikov-Poisson\xspace}
\nc{\preliepoisson}{pre-Lie Poisson\xspace}
\nc{\transpreliepoisson}{transposed pre-Lie poisson\xspace}

\nc{\calb}{\mathcal{B}}
\nc{\rk}{\mathrm{r}}
\newcommand{\g}{\mathfrak g}
\newcommand{\h}{\mathfrak h}
\newcommand{\pf}{\noindent{$Proof$.}\ }
\newcommand{\frkg}{\mathfrak g}
\newcommand{\frkh}{\mathfrak h}
\newcommand{\Id}{\rm{Id}}
\newcommand{\gl}{\mathfrak {gl}}
\newcommand{\ad}{\mathrm{ad}}
\newcommand{\add}{\frka\frkd}
\newcommand{\frka}{\mathfrak a}
\newcommand{\frkb}{\mathfrak b}
\newcommand{\frkc}{\mathfrak c}
\newcommand{\frkd}{\mathfrak d}
\newcommand {\comment}[1]{{\marginpar{*}\scriptsize\textbf{Comments:} #1}}

\nc{\vspa}{\vspace{-.1cm}}
\nc{\vspb}{\vspace{-.2cm}}
\nc{\vspc}{\vspace{-.3cm}}
\nc{\vspd}{\vspace{-.4cm}}
\nc{\vspe}{\vspace{-.5cm}}


\nc{\disp}[1]{\displaystyle{#1}}
\nc{\bin}[2]{ (_{\stackrel{\scs{#1}}{\scs{#2}}})}  
\nc{\binc}[2]{ \left (\!\! \begin{array}{c} \scs{#1}\\
    \scs{#2} \end{array}\!\! \right )}  
\nc{\bincc}[2]{  \left ( {\scs{#1} \atop
    \vspace{-.5cm}\scs{#2}} \right )}  
\nc{\ot}{\otimes}
\nc{\sot}{{\scriptstyle{\ot}}}
\nc{\otm}{\overline{\ot}}
\nc{\ola}[1]{\stackrel{#1}{\la}}

\nc{\scs}[1]{\scriptstyle{#1}} \nc{\mrm}[1]{{\rm #1}}

\nc{\dirlim}{\displaystyle{\lim_{\longrightarrow}}\,}
\nc{\invlim}{\displaystyle{\lim_{\longleftarrow}}\,}

\nc{\bfk}{{\bf k}} \nc{\bfone}{{\bf 1}}
\nc{\rpr}{\circ}
\nc{\dpr}{{\tiny\diamond}}
\nc{\rprpm}{{\rpr}}

\nc{\mmbox}[1]{\mbox{\ #1\ }} \nc{\ann}{\mrm{ann}}
\nc{\Aut}{\mrm{Aut}} \nc{\can}{\mrm{can}}
\nc{\twoalg}{{two-sided algebra}\xspace}
\nc{\colim}{\mrm{colim}}
\nc{\Cont}{\mrm{Cont}} \nc{\rchar}{\mrm{char}}
\nc{\cok}{\mrm{coker}} \nc{\dtf}{{R-{\rm tf}}} \nc{\dtor}{{R-{\rm
tor}}}
\renewcommand{\det}{\mrm{det}}
\nc{\depth}{{\mrm d}}
\nc{\End}{\mrm{End}} \nc{\Ext}{\mrm{Ext}}
\nc{\Fil}{\mrm{Fil}} \nc{\Frob}{\mrm{Frob}} \nc{\Gal}{\mrm{Gal}}
\nc{\GL}{\mrm{GL}} \nc{\Hom}{\mrm{Hom}} \nc{\hsr}{\mrm{H}}
\nc{\hpol}{\mrm{HP}}  \nc{\id}{\mrm{id}} \nc{\im}{\mrm{im}}

\nc{\incl}{\mrm{incl}} \nc{\length}{\mrm{length}}
\nc{\LR}{\mrm{LR}} \nc{\mchar}{\rm char} \nc{\NC}{\mrm{NC}}
\nc{\mpart}{\mrm{part}} \nc{\pl}{\mrm{PL}}
\nc{\ql}{{\QQ_\ell}} \nc{\qp}{{\QQ_p}}
\nc{\rank}{\mrm{rank}} \nc{\rba}{\rm{RBA }} \nc{\rbas}{\rm{RBAs }}
\nc{\rbpl}{\mrm{RBPL}}
\nc{\rbw}{\rm{RBW }} \nc{\rbws}{\rm{RBWs }} \nc{\rcot}{\mrm{cot}}
\nc{\rest}{\rm{controlled}\xspace}
\nc{\rdef}{\mrm{def}} \nc{\rdiv}{{\rm div}} \nc{\rtf}{{\rm tf}}
\nc{\rtor}{{\rm tor}} \nc{\res}{\mrm{res}} \nc{\SL}{\mrm{SL}}
\nc{\Spec}{\mrm{Spec}} \nc{\tor}{\mrm{tor}} \nc{\Tr}{\mrm{Tr}}
\nc{\mtr}{\mrm{sk}}

\nc{\ab}{\mathbf{Ab}} \nc{\Alg}{\mathbf{Alg}}

\nc{\BA}{{\mathbb A}} \nc{\CC}{{\mathbb C}} \nc{\DD}{{\mathbb D}}
\nc{\EE}{{\mathbb E}} \nc{\FF}{{\mathbb F}} \nc{\GG}{{\mathbb G}}
\nc{\HH}{{\mathbb H}} \nc{\LL}{{\mathbb L}} \nc{\NN}{{\mathbb N}}
\nc{\QQ}{{\mathbb Q}} \nc{\RR}{{\mathbb R}} \nc{\BS}{{\mathbb{S}}} \nc{\TT}{{\mathbb T}}
\nc{\VV}{{\mathbb V}} \nc{\ZZ}{{\mathbb Z}}


\nc{\calao}{{\mathcal A}} \nc{\cala}{{\mathcal A}}
\nc{\calc}{{\mathcal C}} \nc{\cald}{{\mathcal D}}
\nc{\cale}{{\mathcal E}} \nc{\calf}{{\mathcal F}}
\nc{\calfr}{{{\mathcal F}^{\,r}}} \nc{\calfo}{{\mathcal F}^0}
\nc{\calfro}{{\mathcal F}^{\,r,0}} \nc{\oF}{\overline{F}}
\nc{\calg}{{\mathcal G}} \nc{\calh}{{\mathcal H}}
\nc{\cali}{{\mathcal I}} \nc{\calj}{{\mathcal J}}
\nc{\call}{{\mathcal L}} \nc{\calm}{{\mathcal M}}
\nc{\caln}{{\mathcal N}} \nc{\calo}{{\mathcal O}}
\nc{\calp}{{\mathcal P}} \nc{\calq}{{\mathcal Q}} \nc{\calr}{{\mathcal R}}
\nc{\calt}{{\mathcal T}} \nc{\caltr}{{\mathcal T}^{\,r}}
\nc{\calu}{{\mathcal U}} \nc{\calv}{{\mathcal V}}
\nc{\calw}{{\mathcal W}} \nc{\calx}{{\mathcal X}}
\nc{\CA}{\mathcal{A}}

\nc{\fraka}{{\mathfrak a}} \nc{\frakB}{{\mathfrak B}}
\nc{\frakb}{{\mathfrak b}} \nc{\frakd}{{\mathfrak d}}
\nc{\oD}{\overline{D}}
\nc{\frakF}{{\mathfrak F}} \nc{\frakg}{{\mathfrak g}}
\nc{\frakm}{{\mathfrak m}} \nc{\frakM}{{\mathfrak M}}
\nc{\frakMo}{{\mathfrak M}^0} \nc{\frakp}{{\mathfrak p}}
\nc{\frakS}{{\mathfrak S}} \nc{\frakSo}{{\mathfrak S}^0}
\nc{\fraks}{{\mathfrak s}} \nc{\os}{\overline{\fraks}}
\nc{\frakT}{{\mathfrak T}}
\nc{\oT}{\overline{T}}
\nc{\frakX}{{\mathfrak X}} \nc{\frakXo}{{\mathfrak X}^0}
\nc{\frakx}{{\mathbf x}}
\nc{\frakTx}{\frakT}      
\nc{\frakTa}{\frakT^a}        
\nc{\frakTxo}{\frakTx^0}   
\nc{\caltao}{\calt^{a,0}}   
\nc{\ox}{\overline{\frakx}} \nc{\fraky}{{\mathfrak y}}
\nc{\frakz}{{\mathfrak z}} \nc{\oX}{\overline{X}}


\title[Solvability and nilpotency of transposed Novikov-Poisson algebras]{Solvability and nilpotency of transposed Novikov-Poisson algebras}

\author{Jiarou Jin}
\address{School of Mathematics, Hangzhou Normal University,
Hangzhou, 311121, China}
\email{jrjin@stu.hznu.edu.cn}

\author{Yanyong Hong (corresponding author)}
\address{School of Mathematics, Hangzhou Normal University,
Hangzhou, 311121, China}
\email{yyhong@hznu.edu.cn}



\begin{abstract}

In this paper, we develop the theory of nilpotency and solvability for transposed Novikov-Poisson algebras. We first establish several equivalent conditions for a dialgebra to be nilpotent, and show that the lower central series of a transposed Novikov-Poisson algebra $P$  admits a simplified form. We then prove that
$P$ is solvable if and only if it is right nilpotent, and also if and only if $P^2$ is nilpotent. Moreover, we show that nilpotency (respectively, solvability) of a transposed Novikov-Poisson algebra is equivalent to nilpotency (respectively, solvability) of both its underlying commutative associative algebra and its underlying Novikov algebra. Finally, we prove that It\^{o}'s  theorem holds for transposed Novikov-Poisson algebras.
\end{abstract}

\maketitle




\allowdisplaybreaks

\section{Introduction}
Novikov algebras first emerged in the study of Hamiltonian
operators in the formal variational calculus \cite{GD1, GD2} and Poisson brackets of hydrodynamic type \cite{BN}. As shown in \cite{X1}, they also correspond to a class of Lie conformal algebras that describe the singular part of operator product expansion of chiral fields in conformal field
theory \cite{K1}.
Recall that a {\bf (left) Novikov algebra} $(A,\circ)$ is a vector space $A$ equipped with a binary operation $\circ:A \otimes A \to A$ satisfying the following compatibility conditions:
\begin{eqnarray}
&\label{Novikov1}(x\circ y)\circ z=(x\circ z)\circ y,\\
&\label{Novikov2}(x\circ y)\circ z-x\circ (y\circ z)=(y\circ x)\circ z-y\circ (x\circ z)\;\;\;\text{for all $x$, $y$, $z\in A$.}
\end{eqnarray}
\delete{Let $B$ be a vector space with a binary operation $\diamond$. If $(B, \circ)$ is a Novikov algebra with $a\circ b:=b\diamond a$ for all $a$, $b\in B$, then $(B, \diamond)$ is called a {\bf right Novikov algebra}.} By Eq. (\ref{Novikov2}), Novikov algebras are also an important subclass of pre-Lie algebras
which have close relationships with many fields in mathematics and physics such as convex homogeneous cones, affine manifolds and affine structures on Lie groups, deformation of associative algebras, vertex algebras and so on.

The notion of a transposed Novikov-Poisson algebra was introduced in~\cite{transposed N-P}. Specifically, a triple $(P,\cdot,\circ)$ is called a {\bf transposed Novikov-poisson algebra} if $(P,\cdot)$ is a commutative associative algebra, $(P,\circ)$ is a Novikov algebra, and the following compatibility conditions hold for all $x$, $y$, $z\in P$:
\begin{eqnarray}
&\label{transNP1}\quad(x\cdot y)\circ z=(x\cdot z)\circ y,\\
&\label{transNP2}2z\cdot(x\circ y)=(z\cdot x)\circ y+x\circ(z\cdot y).
\end{eqnarray}
As demonstrated in~\cite{transposed N-P}, the affinization of a transposed Novikov-Poisson algebra yields a transposed Poisson algebra, a structure originally introduced in~\cite{Bai}. \delete{Note that the affinization of Novikov algebras was introduced in~\cite{BN}, and this construction yields many important infinite-dimensional Lie algebras, such as the Witt algebra, the centerless Heisenberg-Virasoro algebra, the centerless Schr\"odinger-Virasoro algebra and so on.} Moreover, transposed Poisson algebras  arise naturally from transposed Novikov-Poisson algebras by taking the commutator Lie algebra of the Novikov algebra. In addition, a natural transposed Poisson algebra structure exists on the tensor product of a transposed Novikov-Poisson algebra and a right differential Novikov-Poisson algebra. We note that differential Novikov-Poisson algebras, introduced in~\cite{BCZ}, form a subclass of the Novikov-Poisson algebras studied in~\cite{Xu1}. Beyond these constructions, transposed Novikov-Poisson algebras have many important properties, such as closure under tensor products and a close connection with $\frac{1}{2}$-derivations of the underlying Novikov algebras, the latter providing a method for constructing transposed Novikov-Poisson algebra structures on certain Novikov algebras.\par

A fundamental structural result established in \cite{transposed N-P} is that, over an algebraically closed field {\bf k} of characteristic zero or $p>2$,  every finite-dimensional non-trivial simple transposed Novikov-Poisson algebra is one-dimensional. This fact naturally directs attention to the study of nilpotent and solvable transposed Novikov-Poisson algebras. Simultaneously, the work~\cite{Poisson algebras-non-associative algebras} reformulated Poisson algebras in terms of non-associative algebras via polarization and depolarization, making it natural to involve both the associative and Lie products when defining nilpotency and solvability. Accordingly, the study of nilpotent and solvable Poisson algebras has progressed considerably in recent years~\cite{On the solvable Poisson algebras, Frattini of dialgebras}. Parallel developments have also emerged for other dialgebraic structures, including generalized Poisson algebras~\cite{generalized Poisson}, transposed Poisson algebras~\cite{niltransposed}, and Poisson $n$-Lie algebras~\cite{nilpoissonnlie}. In the specific context of Novikov algebras, substantial results are already available. E. Zelmanov proved that for a left-nilpotent finite-dimensional right Novikov algebra $N$, $N^2$ is nilpotent~\cite{Zelmanov Novikov}. I. Shestakov and Z. Zhang established the equivalence between solvability and right nilpotency in left Novikov algebras~\cite{Nilpotent Novikov}. D. Towers showed that every left nilpotent Novikov algebra is nilpotent~\cite{Frattini Novikov}. These developments naturally motivate our investigation into solvability and nilpotency of transposed Novikov-Poisson algebras.\par

It\^{o}'s theorem, originating in group theory  \cite{itogroup}, states that if  a group $G$ is the product $G=A\cdot B$ of two abelian subgroups $A$ and $B$, then $G$ is metabelian. This result has since been extended to a variety of algebraic settings, including Lie algebras~\cite{itolie1,itolie2}, Leibniz algebras~\cite{itolei}, Novikov algebras~\cite{Nilpotent Novikov}, associative algebras~\cite{itoass}, as well as some dialgebras, for example, non-associative Poisson algebras~\cite{itopoisson} and generalized Poisson algebras~\cite{itogenpoisson}.\par

 The paper is organized as follows. In Section 2, we recall the relevant definitions of nilpotency and solvability for dialgebras and establish several equivalent characterizations of nilpotent dialgebras. In particular, we show that the lower central series of a transposed Novikov-Poisson algebra admits a simplified form, which will be used in subsequent sections. In Section 3, we investigate the relationships among nilpotency, right nilpotency, and solvability for transposed Novikov-Poisson algebras. We prove that a transposed Novikov-Poisson algebra $(P,\cdot ,\circ)$ is solvable if and only if it is right nilpotent, and also if and only if $P^2$ is nilpotent. We further show that nilpotency (respectively, solvability) of a transposed Novikov-Poisson algebra is equivalent to nilpotency (respectively, solvability) of both its underlying commutative associative algebra and its underlying Novikov algebra.
 In Section 4, we provide equivalent conditions for metatriviality of a transposed Novikov-Poisson algebra, and show that It\^{o}'s theorem holds for transposed Novikov-Poisson algebras.

\noindent{\bf Notations.} Throughout this paper, let ${\bf k}$ be a field whose characteristic is not equal to $2$, and denote by $\mathbb{C}$ the set of complex numbers.

\vspace{-.2cm}
\section{Related Properties of Transposed Novikov-Poisson Algebras}
Recall that a {\bf dialgebra} $(P,\cdot, \circ)$ is a vector space $P$ with two binary operations $\cdot$ and $\circ$. Note that associative dialgebras were originally defined by Loday in the 1990s~\cite{assodialgebradefi}. A \textbf{subalgebra} of a dialgebra $(P,\cdot, \circ)$ is a linear subspace \(A\) satisfying
$A \cdot A + A\circ A \subseteq A$. A subalgebra \(I\) of $(P,\cdot, \circ)$ is an \textbf{ideal} if
$I \cdot P + P \cdot I + I\circ P + P\circ I \subseteq I$. Note that a transposed Novikov-Poisson algebra is a dialgebra. Furthermore, for each $x\in P$, define $P_x, P'_x, Q_x$ and $Q'_x\in\operatorname{End}(P)$ by $P_x(y)=x\cdot y$, $p'_x(y)=y\cdot x$,$Q_x(y)=x\circ y$ and $Q'_x(y)=y\circ x$ respectively for all $x,y\in P$.
Given a transposed Novikov-Poisson algebra $(P,\cdot,\circ )$. We use the notation $P_A$ and $P_N$ to denote the associative algebra structure $(P,\cdot)$ and the Novikov algebra structure $(P,\circ)$ respectively.

\begin{defi}~\cite{dialgebra nilsol}
Let $(P,\cdot,\circ)$ be a dialgebra and $A$ be a subalgebra of $P$. The \textbf{derived series} of $A$ is defined by
\begin{eqnarray*}
&&A^{(0)} := A, \qquad A^{(n+1)} := A^{(n)}\!\cdot\! A^{(n)} + A^{(n)}\circ A^{(n)}\; \quad\text{for all}\:n\geq 0.
\end{eqnarray*}
The dialgebra $(A,\cdot,\circ)$ is called \textbf{solvable} if there exists $m \ge 0$ such that $A^{(m)} = 0$.\par
The \textbf{lower central series} of $A$ is the sequence of $A$ given by
\begin{eqnarray*}
A^{1} := A, \qquad
A^{n} := \sum_{i=1}^{n-1} \bigl( A^{i}\!\cdot\! A^{n-i} + A^{i}\circ A^{n-i} \bigr) \quad\text{for all}\:n\geq 2.
\end{eqnarray*}
The dialgebra $(A,\cdot,\circ)$ is called \textbf{nilpotent} if there exists $m \ge 1$ such that $A^{m} = 0$.
\end{defi}

Next, we provide several equivalent conditions for a finite-dimensional dialgebra to be nilpotent, thereby extending the familiar characterizations of nilpotent algebras ~\cite{nilalgebra1}.

\begin{pro}\label{dialgebra nil4}
Let $(P, \cdot, \circ)$ be a finite-dimensional dialgebra. The following statements are equivalent:
\begin{enumerate}
\item \label{i}\(P\) is nilpotent;
\item \label{ii}If \(B \neq P\) is a subalgebra of \(P\), then there exists an element \(x \notin B\) such that \(x\cdot y,y\cdot x,x\circ y,y\circ x \in B\), for all \(y \in P\);
\item \label{iii} There is a basis \(\{x_1, \dots, x_n\}\) for \(P\) such that the matrices representing \(P_x, P'_x,Q_x\) and \(Q'_x\) with respect to this basis are strictly triangular, for all \(x \in P\);
\item \label{iv} There is a chain of ideals of \(P\) as follows:
 \begin{eqnarray*}
&&0 = P_{(0)} \subset P_{(1)} \subset \cdots \subset P_{(n)} = P,
 \end{eqnarray*}
 where \(\dim P_{(i)} = i\) and \(P\cdot  P_{(i)} +P_{(i)}\cdot P+ P_{(i)} \circ P+P\circ P_{(i)} \subseteq P_{(i-1)}\), for each $i=1,2,...,n$.
\end{enumerate}
\end{pro}
\begin{proof}
$\eqref{i} \Rightarrow \eqref{ii}$: Let \(B \neq P\) be a subalgebra of \(P\), where \(P\) is nilpotent. Then there exists a minimal integer \(r \geq 2\) such that all products (including Novikov product and associate product) of \(r\) elements of \(P\) belong to \(B\). So there exists an element $z$, which is the product of \((r-1)\) elements and does not belong to \(B\). Then we have \(x\cdot z,z\cdot x ,x\circ z, z\circ x \in B\) for all $x\in P$.\\
$\eqref{ii}\Rightarrow\eqref{iii}$: Taking \(B = 0\) in (b) we have that there is a non-zero element \(x_1 \in P\) such that \(x_1\cdot y = y\cdot x_1=x_1\circ y=y\circ x_1 = 0\) for all \(y \in P\). Now taking \(B = {\bf k}x_1\), which is an ideal of \(P\), we see that there is an element \(x_2 \in P\), \(x_2 \notin {\bf k}x_1\) such that \(x_2\cdot y,y\cdot x_2,x_2\circ y,y\circ x_2 \in {\bf k}x_1\), for all \(y \in P\). Then \({\bf k}x_1 + {\bf k}x_2\) is an ideal of \(P\). Continuing in this way we construct \(n\) linearly independent elements \(x_1, \dots, x_n\) such that \(x_i \notin {\bf k}x_1 + \dots + {\bf k}x_{i-1}\), but \(x_i\cdot y,y\cdot x_i, x_i\circ y, y\circ x_i \in {\bf k}x_1 + \dots + {\bf k}x_{i-1}\) for all \(y \in P\), \(2 \leq i \leq n\). This is the required basis.\\
$\eqref{iii}\Rightarrow \eqref{iv}$: Let \(P_{(i)} = {\bf k}x_1 + \dots + {\bf k}x_i\) (\(1 \leq i \leq n\)).\\
$\eqref{iv}\Rightarrow \eqref{i}$: It is straightforward.
\end{proof}
Similar to left nilpotency and right nilpotency for Novikov algebras, we introduce the following definitions.
\begin{defi}
 Let $(P,\cdot,\circ)$ be a transposed Novikov-Poisson algebra and $A$ be a subalgebra of $P$. Define \(A_L^1 = A\) and \(A_L^n = A_L^{n-1}\cdot A +A_L^{n-1}\circ A\) for \(n \geq 2\).
Then \(A\) is called {\bf right nilpotent} of index \(n\) if \(n\) is the minimal integer such that \(A_L^n = (0)\). Define \(A_R^1 = A\) and \(A_R^n = A\cdot A_R^{n-1} +A\circ A_R^{n-1}\) for \(n \geq 2\).
Then \(A\) is called {\bf left nilpotent} of index \(m\) if \(m\) is the minimal integer such that \(A_R^m = (0)\).
\end{defi}

\begin{pro}\label{idealpro1}
Let $(P,\cdot,\circ)$ be a transposed Novikov-Poisson algebra and $B,C$ be two ideals of $P$. Then $B\circ C$ is also an ideal of $P$.
\end{pro}

\begin{proof}
By \cite{simple Novikov},  $B\circ C$ is an ideal of $P_N$. By Eq.~\eqref{transNP2}, we obtain
\begin{eqnarray*}
&&P\cdot (B\circ C)\subseteq (B\cdot P)\circ C+B\circ (C\cdot P)\subseteq B\circ C.
\end{eqnarray*}
This completes the proof.
\end{proof}

\begin{pro}\label{idealpro2}
Let $(P,\cdot,\circ)$ be a transposed Novikov-Poisson algebra and $B,C$ be two ideals of $P$. Then $B\cdot C+B\circ C$ is also an ideal of $P$. In particular, $P^m$, $P^{(n)}$, $P_L^m$, $P_R^m$ are ideals, for all $m\ge 1$, $n\ge0$.
\end{pro}
\begin{proof}
It is straightforward.
\end{proof}

\begin{pro}~\cite[Proposition 2.20]{transposed N-P}
Let $(P,\cdot,\circ)$ be a {\transNovikovpoisson algebra}. Then the following identity holds for all $x,y,z\in P$:
\begin{eqnarray}
\label{Tid1}&&(x\circ y)\cdot z=(x\circ z)\cdot y.
\end{eqnarray}
\end{pro}

The lower central series of transposed Novikov-Poisson algebras can be given in a more concise form.
\begin{pro}\label{nildefi1}
Let $(P,\cdot,\circ)$ be a transposed Novikov-Poisson algebra and $A$ be a subalgebra of $P$. Then we have $A^{k+1}=A\cdot A^{k}+A\circ A^{k}+A^k\circ A$ for $k\geq 1$.
\end{pro}

\begin{proof}
Since $A \cdot A^{k}+A\circ A^{k}+A^{k}\circ A \subseteq A^{k+1}$ is obvious, we only need to prove that
\begin{eqnarray*}
A^{k+1}\subseteq A^{k}\cdot A+A\circ A^{k}+A^{k}\circ A\qquad\text{for all }k\geq 1.
\end{eqnarray*}
We proceed by induction on $k$.
If $k=1$, the statement is obvious.
Assume that $A^{n+1}\subseteq A^{n}\cdot A+A\circ A^{n}+A^{n}\circ A$ holds for all $1\leq n\leq k-1$.
Now consider $A^{k+1}$. It suffices to show
\[
A^{i}\cdot A^{k+1-i}+A^{i}\circ A^{k+1-i}+A^{k+1-i}\circ A^{i} \subseteq A^{k}\cdot A+A\circ A^{k}+A^{k}\circ A\qquad\text{for all }1\leq i\leq k.
\]
We prove this by a second induction on $i$.
For $i=1$, the inclusion is clear.
Suppose that $A^{j}\cdot A^{k+1-j}+A\circ A^{k+1-j}+A^{k+1-j}\circ A\subseteq A^{k}\cdot A+A\circ A^{k}+A^{k}\circ A$ holds for all $1\leq j\leq i$. Consider $A^{i+1}\cdot A^{k-i}+A^{i+1}\circ A^{k-i}+A^{k-i}\circ A^{i+1}$.
Since $i+1\leq k$, applying the outer inductive hypothesis, and by Eqs.~\eqref{transNP1},~\eqref{transNP2} and~\eqref{Tid1}, we obtain
\begin{eqnarray*}
A^{i+1}\cdot A^{k-i}&=&\big(A\circ A^{i}+ A^{i}\circ A+A\cdot A^{i}\big) \cdot A^{k-i}\\
&\subseteq& (A\circ A^{k-i})\cdot A^{i}+(A^i\circ A^{k-i})\cdot A+A\cdot (A^i\cdot A^{k-i})\\
&\subseteq& A^{k}\cdot A+A\circ A^{k}+A^{k}\circ A,
\end{eqnarray*}
\begin{eqnarray*}
A^{i+1}\circ A^{k-i}&=&\big(A\circ A^{i}+ A^{i}\circ A+A\cdot A^{i}\big) \circ A^{k-i} \\
&\subseteq& (A\circ A^{k-i})\circ A^{i}+(A^i\circ A^{k-i})\circ A+ (A\cdot A^{k-i})\circ A^i \\
&\subseteq& A^{k}\cdot A+A\circ A^{k}+A^{k}\circ A,
\end{eqnarray*}
and
\begin{eqnarray*}
A^{k-i}\circ A^{i+1}&=&A^{k-i}\circ \big(A\circ A^{i}+ A^{i}\circ A+A\cdot A^{i}\big)  \\
&\subseteq& (A^{k-i}\circ A)\circ A^i+(A\circ A^{k-i})\circ A^{i}+A\circ (A^{k-i}\circ A^i)+(A^{k-i}\circ A^i)\circ A\\
&+& (A^i\circ A^{k-i})\circ A+A^i\circ (A^{k-i}\circ A)+A\cdot (A^{k-i}\circ A^i)+(A\cdot A^{k-i})\circ A^i \\
&\subseteq& A^{k}\cdot A+A\circ A^{k}+A^{k}\circ A.
\end{eqnarray*}
This completes the inner induction, and consequently the outer induction. This completes the proof.
\end{proof}

\section{The nilpotency and solvability of transposed Novikov-Poisson algebras}
In this section, we will introduce the properties of nilpotency, left and right nilpotency, and solvability of transposed Novikov-Poisson algebras.

First, we consider the relationship between the right nilpotency and solvability of transposed Novikov-Poisson algebras. 
\begin{lem}\label{right nil-P^2nil}
Let $(P,\cdot,\circ )$ be a transposed Novikov-Poisson algebra. If \( P \) is right nilpotent, then \( P^2 \) is nilpotent, in particular, \( P \) is solvable.
\end{lem}

\begin{proof}
We claim that $(P^2)^m\subseteq P_L^{m+1}$ for all $m\geq 1$. Obviously, it holds for $m=1$.
Suppose that $(P^2)^k\subseteq P_L^{k+1}$ holds for all $1\leq k\leq m$. Consider $(P^2)^{m+1}=P^2\cdot (P^2)^m+P^2\circ (P^2)^m+(P^2)^m\circ P^2$. By the hypothesis and Proposition~\ref{idealpro2}, we obtain
\begin{eqnarray*}
(P^2)^{m+1}&\subseteq& P_L^{m+1}\cdot P^2+P_L^{m+1}\circ P^2+P^2\circ P_L^{m+1}\\
&\subseteq& P_L^{m+2}+(P\cdot P)\circ P_L^{m+1}+(P\circ P)\circ P_L^{m+1}\\
&\subseteq&P_L^{m+2}+(P\cdot P_L^{m+1})\circ P+(P\circ P_L^{m+1})\circ P\\
&\subseteq& P_L^{m+2}.
\end{eqnarray*}
This completes the induction. Then this conclusion follows directly.
\end{proof}

\begin{lem}\label{sol-rightnil}
Let $(P,\cdot,\circ )$ be a transposed Novikov-Poisson algebra. For all \( m \geq 0 \) and \( n \geq 0 \), we have \( (P^{(m)})_L^{3^n} \subseteq P^{(m+n)} \).
\end{lem}
\begin{proof}
The proof is similar to that in~\cite{Nilpotent Novikov}.
We first claim that $\big(P^{(m)}\big)_L^{q}\subseteq \big(P^{(m)}\big)_L^{q-3}\cdot P^{(m+1)}+\big(P^{(m)}\big)_L^{q-3}\circ P^{(m+1)}$ holds for all integer $q\ge4$. We prove it by induction on $q$. If $q=4$, it is clear. Note that $\big(P^{(m)}\big)_L^{q+1}=\big(P^{(m)}\big)_L^{q}\cdot P^{(m)}+\big(P^{(m)}\big)_L^{q}\circ P^{(m)}$. By the hypothesis, we obtain
\begin{eqnarray*}
\big(P^{(m)}\big)_L^{q+1}&\subseteq&\big(\big(P^{(m)}\big)_L^{q-3}\cdot P^{(m+1)}+\big(P^{(m)}\big)_L^{q-3}\circ P^{(m+1)}\big)\cdot P^{(m)}\\
&+&\big(\big(P^{(m)}\big)_L^{q-3}\cdot P^{(m+1)}+\big(P^{(m)}\big)_L^{q-3}\circ P^{(m+1)}\big)\circ P^{(m)}\\
&\subseteq&\big(\big(P^{(m)}\big)_L^{q-3}\cdot P^{(m)}+\big(P^{(m)}\big)_L^{q-3}\circ P^{(m)}\big)\cdot P^{(m+1)}\\
&+&\big(\big(P^{(m)}\big)_L^{q-3}\cdot P^{(m)}+\big(P^{(m)}\big)_L^{q-3}\circ P^{(m)}\big)\circ P^{(m+1)}\\
&\subseteq& \big(P^{(m)}\big)_L^{q-2}\cdot P^{(m+1)}+\big(P^{(m)}\big)_L^{q-2}\circ P^{(m+1)}.
\end{eqnarray*}
Hence, we complete the induction, and the claim.\par
Next, we prove $\big(P^{(m)}\big)_L^{3n}\subseteq \big(P^{(m+1)}\big)_L^n$ for all $n\ge 1$, $m\ge 0$. We prove it by induction on $n$. When $n=1$, it is clear. By the hypothesis, we obtain
\begin{eqnarray*}
(P^{(m)})_L^{3n+3}&\subseteq& (P^{(m)})_L^{3n}\cdot P^{(m+1)}+  (P^{(m)})_L^{3n}\circ P^{(m+1)}\\
&\subseteq &(P^{(m+1)})_L^{n}\cdot P^{(m+1)}+  (P^{(m+1)})_L^{n}\circ P^{(m+1)}\\
&\subseteq&(P^{(m+1)})_L^{n+1}.
\end{eqnarray*}
This completes the induction. Finally, we prove $(P^{(m)})_L^{3^n} \subseteq P^{(m+n)}$ by induction on $n$. Then we obtain
\begin{eqnarray*}
&&(P^{(m)})_L^{3^{n+1}}\subseteq (P^{(m+1)})_L^{3^{n}}\subseteq P^{(m+n+1)}.
\end{eqnarray*}
This completes the conclusion.
\end{proof}

\begin{thm}\label{nil-soltrans}
Let $(P,\cdot, \circ )$ be a transposed Novikov-Poisson algebra. Then the following statements are equivalent:
\begin{enumerate}
\item[(i)] \( P \) is right nilpotent;
\item[(ii)] \( P^2 \) is nilpotent;
\item[(iii)] \( P \) is solvable.
\end{enumerate}
\end{thm}

\begin{proof}
It is straightforward by Lemmas~\ref{right nil-P^2nil} and~\ref{sol-rightnil}.
\end{proof}

\begin{ex}\label{ex1}
Let $P$ be a 3-dimensional vector space over $\mathbb{C}$ with a basis $\{e_1, e_2,e_3\}$.
By~\cite{transitive Novikov}, $(P, \circ)$ is a transitive (right nilpotent) Novikov algebra with the non-zero products given by
\begin{eqnarray*}
&&e_2 \circ e_2 = e_3\circ e_3=e_1.
\end{eqnarray*}
It is easy to see that it is solvable. By the method~\cite{transposed N-P}, there is a non-trivial transposed Novikov-Poisson algebra on it,
where the non-zero commutative associative products are given by
\begin{eqnarray*}
 && e_2 \cdot e_2 = e_2\cdot e_3=e_3\cdot e_2=e_3\cdot e_3=e_1.
\end{eqnarray*}
By a straightforward computation, we obtain $P^{(1)}=P^2=P_L^2=\mathbb{C}e_1$ and $(P^2)^1=P^{(2)}=P_L^3=0$. Hence, $P$ is solvable and right nilpotent, and $P^2$ is nilpotent.
\end{ex}

Next, we consider the relationship between the nilpotency and solvability of a transposed Novikov-Poisson algebra and those of its underlying commutative associative algebra and Novikov algebra.
\begin{lem}\label{minimalideal-Annlem1}
Let $B$ be a minimal ideal of a transposed Novikov-Poisson algebra $P$. If $N$  is an associative and Novikov nilpotent ideal of $P$, then $B\subseteq Ann_P(N)=\{ x \mid x\circ N=N\circ x=x\cdot N=0\}$.
\end{lem}

\begin{proof}
By Proposition~\ref{idealpro1}, $B\circ N$ is an ideal of $P$. Since $B\circ N\subseteq B$, the minimality of $B$ implies that either $B\circ N=B$ or $B\circ N=0$. Because $N$ is Novikov nilpotent, we must have $B\circ N=0$. Similarly, we obtain $N\circ B=0$.\par
Moreover, by Eqs.~\eqref{transNP1} and~\eqref{transNP2}, we have
\begin{eqnarray*}
(B\cdot N)\cdot P = B\cdot(N\cdot P)  \subseteq B\cdot N
\end{eqnarray*}
and
\begin{eqnarray*}
(B\cdot N)\circ P+P\circ (B\cdot N) \subseteq(B\cdot P)\circ N+B\cdot (P\circ N)+(P\cdot B)\circ N\subseteq B\cdot N+B\circ N=B\cdot N.
\end{eqnarray*}
It follows that $B\cdot N$ is also an ideal of $P$ contained in $B$. Thus, by the minimality of $B$, we have either $B\cdot N=0$ or $B\cdot N=B$. Since $N$ is associative nilpotent, we conclude $B\cdot N=0$. Hence, we obtain $B\subseteq \operatorname{Ann}_P(N)$.
\end{proof}

\begin{thm}\label{engelthm1}
Let $(P,\cdot,\circ)$ be a finite-dimensional transposed Novikov-Poisson algebra. Then $P_A$ and $P_N$ are nilpotent if and only if $P$ is nilpotent.
\end{thm}

\begin{proof}
If $P$ is nilpotent, then $P_A$ and $P_N$ are nilpotent clearly.\par
Conversely, assume that $P_A$ and $P_N$ are nilpotent. We prove that $P$ is nilpotent by induction on $\dim(P)$. If $\operatorname{dim}(P)=1$, then $P\circ P=P\cdot P=0$ and the conclusion holds trivially. Assume that the conclusion holds for all transposed Novikov-Poisson algebras of dimensions at most $k$. Suppose that $\operatorname{dim}(P)=k+1$.\par
If $P\circ P=0$, then $P=P_A$. There is nothing to prove. Otherwise, by Proposition~\ref{idealpro1}, there exists a minimal ideal $B$ of $P$ contained in $P\circ P$. Consider the quotient $P/B$. By the induction hypothesis, $P/B$ is nilpotent. Hence there exists a non-negative integer $m$ such that $P^m\subseteq B$. By the proof of Lemma~\ref{minimalideal-Annlem1}, we have $B\subseteq \operatorname{Ann}_P(P)$. Therefore, we obtain
\begin{eqnarray*}
P^{m+1}= P\cdot P^m+ P^{m}\circ P+P\circ P^m\subseteq P\cdot B+ P\circ B+B\circ P=0.
\end{eqnarray*}
This completes the proof.
\end{proof}

By ~\cite{Frattini Novikov}, any finite-dimensional Novikov algebra necessarily has a maximal nilpotent ideal. So we can obtain the following corollary.
\begin{cor}
Let $(P,\cdot,\circ)$ be a finite-dimensional transposed Novikov-Poisson algebra. Then there exists a maximal nilpotent ideal of $P$.
\end{cor}
\begin{proof}
It is straightforward by Theorem~\ref{engelthm1}.
\end{proof}
Then we have the following definition.
\begin{defi}
For a finite-dimensional transposed Novikov-Poisson algebra $(P,\cdot,\circ)$, we denote by $Nil(P)$ the {\bf nilpotent radical} of $P$, that is, the maximal nilpotent ideal of $P$. Similarly, let $Nil_A(P)$ be the nilpotent radical of $P_A$, and $Nil_N(P)$ be the nilpotent radical of $P_N$.
\end{defi}

 In fact, Lemma~\ref{minimalideal-Annlem1} and Theorem~\ref{engelthm1} have a right-nilpotent version.
\begin{lem}\label{minimalideal-Annlem2}
Let $B$ be a minimal ideal of a transposed Novikov-Poisson algebra $P$. If $N$ is an associative nilpotent and Novikov right nilpotent ideal of $P$, then $B\subseteq Ann_L(N)=\{x\in P \mid x\cdot N=x\circ N=0\}$.
\end{lem}
\begin{proof}
The proof is similar to that in Lemma~\ref{minimalideal-Annlem1}.
\end{proof}

\begin{thm}\label{engelthm2}
Let $(P,\cdot,\circ )$ be a finite-dimensional transposed Novikov-Poisson algebra. Then $P_A$ is nilpotent and $P_N$ is right nilpotent if and only if $P$ is right nilpotent.
\end{thm}

\begin{proof}
The proof is similar to Theorem~\ref{engelthm1} by Lemma~\ref{minimalideal-Annlem1}.
\end{proof}

\begin{lem}\label{nil-solNovikov}~\cite[Theorem 3.3]{Nilpotent Novikov}
Let $N$ be a Novikov algebra. Then the following statements are equivalent:
\begin{enumerate}
\item[(i)] \( N \) is right nilpotent;
\item[(ii)] \( N^2 \) is nilpotent;
\item[(iii)] \( N \) is solvable.
\end{enumerate}
\end{lem}

\begin{thm}
 Let $(P,\cdot,\circ )$ be a finite-dimensional transposed Novikov-Poisson algebra. Then $P$ is solvable if and only if $P_A$ and $P_N$ are solvable.
\end{thm}

\begin{proof}
The necessity is clear. Now we prove the sufficiency.  Since $P_A$ and $P_N$ are solvable, by Lemma~\ref{nil-solNovikov} and Theorem~\ref{engelthm2}, $P$ is right nilpotent. Hence, $P$ is solvable by Theorem~\ref{nil-soltrans}.
\end{proof}

\begin{cor}
Let $(P,\cdot,\circ )$ be a finite-dimensional solvable transposed Novikov-Poisson algebra. Then $Nil(P)=Nil_N(P)$.
\end{cor}
\begin{proof}
Obviously, we have $Nil(P)\subseteq Nil_N(P)$. By Theorem~\ref{nil-soltrans}, we obtain that $Nil_N(P)\cdot P\subseteq P^2\subseteq Nil(P)\subseteq Nil_N(P)$, which implies that $Nil_N(P)$ is an ideal of $P$. Hence, we have $Nil(P)=Nil_N(P)$.
\end{proof}

\begin{defi}\cite{Frattini of dialgebras}
If \( B, C \) are ideals of a finite-dimensional dialgebra \( P \) with \( C \subset B \), a {\bf chief series} of \( P \) from \( C \) to \( B \) is a series
\( C = C_0 \subset C_1 \subset \cdots \subset C_r = B \), where \( C_i \) are ideals of \( P \) and \( C_{i+1}/C_i \) is a minimal ideal of \( P/C_i \) for
\( 0 \leq i \leq r - 1 \). The factor algebras \( C_{i+1}/C_i \) in this series are called {\bf chief factors}.
\end{defi}

\begin{lem}\label{chief ideal lem}
Let $(P,\cdot,\circ )$ be a finite-dimensional left (resp. right) nilpotent transposed Novikov-Poisson algebra and $B/C$ be a chief factor of $P$. Then $P\cdot B+P\circ B\subseteq C$ (resp. $B\cdot P+B\circ P\subseteq C$).
\end{lem}

\begin{proof}
Here we only prove the left nilpotent case.  The right nilpotent case is similar. By Proposition~\ref{idealpro2}, we obtain that $P\cdot B+P\circ B+C$ is an ideal of $P$. Thus we have $P\cdot B+P\circ B+C\subseteq C$ or $P\cdot B+P\circ B+C=B$. Suppose that the latter holds. Since $P$ is left nilpotent, we have
\begin{eqnarray*}
&&B=P\cdot B+P\circ B+C=P\cdot (P\cdot B+P\circ B)+P\circ (P\cdot B+P\circ B)+C+...=C,
\end{eqnarray*}
which leads to a contradiction. Hence, $P\cdot B+P\circ B\subseteq C$. Similarly, if $P$ is right nilpotent, then $B\cdot P+B\circ P\subseteq C$.
\end{proof}

\begin{thm}\label{weak nil-nil}
Let $(P,\cdot ,\circ )$ be a finite-dimensional transposed Novikov-Poisson algebra. If $P$ is left and right nilpotent, then $P$ is nilpotent.
\end{thm}
\begin{proof}
Let $B/C$ be a chief factor of $P$. Then we have $B\cdot P+B\circ P+P\circ B\subseteq C$. It implies that $\operatorname{dim}(B/C)=1$. So there is a chain of ideals of $P$
 \begin{eqnarray*}
&&0 = P_{(0)} \subset P_{(1)} \subset \cdots \subset P_{(n)} = P,
 \end{eqnarray*}
 where \(\dim P_{(i)} = i\) and \(P\cdot  P_{(i)} + P_{(i)} \circ P+P\circ P_{(i)} \subseteq P_{(i-1)}\). Hence, $P$ is nilpotent by Theorem~\ref{dialgebra nil4}.
\end{proof}

\begin{cor}
Let $(P,\cdot ,\circ )$ be a finite-dimensional transposed Novikov-Poisson algebra. If $P$ is left nilpotent, then $P$ is nilpotent.
\end{cor}
\begin{proof}
It is straightforward by Theorems~\ref{nil-soltrans} and ~\ref{weak nil-nil}.
\end{proof}

\section{It\^{o}'s theorem for transposed Novikov-Poisson algebras} In this section, we will prove that It\^{o}'s theorem remains valid for transposed Novikov-Poisson algebras. Recall that a dialgebra \((P,\cdot,[\cdot,\cdot])\) is called \textbf{abelian} if for all \(a, b \in P\), we have \(a\cdot b=a\circ b = 0\). Based on the general definitions of metabelian groups and Lie algebras, the definition of metatrivial Poisson algebras was given in ~\cite{itopoisson,metadedi}. Therefore, by analogy, we give the corresponding definition for transposed Novikov-Poisson algebras.
\begin{defi}
Let \((P, \cdot, \circ)\) be a transposed Novikov-Poisson algebra. Then $P$ is called \textbf{metatrivial} (or \textbf{strongly metabelian}) if \(P\) is an extension of an abelian transposed Novikov-Poisson algebra by another abelian one, that is, if there exist two abelian transposed Novikov-Poison algebras \( V_1 \), \( V_2 \) and an exact sequence of algebra morphisms
\begin{eqnarray}\label{metaexact}
0 \longrightarrow V_1 \xrightarrow{i} P \xrightarrow{\pi} V_2 \longrightarrow 0 .
\end{eqnarray}
\end{defi}

\begin{pro}\label{metapro}
Let $(P,\cdot ,\circ)$ be a transposed Novikov-Poisson algebra. Then the following statements are equivalent:
\begin{enumerate}
    \item\label{meta1} \( P \) is metatrivial;
    \item \label{meta2}\( P \) is a solvable algebra with \( (c_1*_1c_2)*_2(c_3*_3c_4)= 0 \), where \( c_1,c_2,c_3,c_4 \in P \) and $*_1,*_2,*_3\in \{\cdot,\circ\}$;
    \item \label{meta3}The derived algebra \( P^2 \) is an abelian subalgebra of \( P \).
\end{enumerate}
\end{pro}

\begin{proof}
$\eqref{meta2}\Leftrightarrow \eqref{meta3}$ is obvious.\\
$\eqref{meta1}\Rightarrow \eqref{meta2}$: Assume that ~\eqref{metaexact} is the corresponding exact sequence of $P$. Then we have $c_1*c_2 \in \operatorname{Ker}(\pi) = V_1$ since $\pi(c_1*c_2)=\pi(c_1)*\pi(c_2)\in V_2*V_2=(0)$, where $c_1,c_2\in P $ and $*\in \{\cdot,\circ \}$. Hence, we obtain \( (c_1*_1c_2)*_2(c_3*_3c_4)\in V_1\cdot V_1+V_1\circ V_1= (0) \), where \( c_1,c_2,c_3,c_4 \in P \) and $*_1,*_2,*_3\in \{\cdot,\circ\}$.\\
$\eqref{meta3}\Rightarrow \eqref{meta1}$: If \( P^1 \) is an abelian subalgebra of \( P \), then we have an exact sequence:
\begin{eqnarray*}
0 \longrightarrow V_1 = P^2 \xrightarrow{i} P \xrightarrow{\pi} (P/P^2)=V_2 \longrightarrow 0,
\end{eqnarray*}
where \( i \) is the inclusion map and \( \pi \) the canonical projection.
\end{proof}

\begin{thm} [Ito's theorem for transposed Novikov-Poisson algebras]\label{ito}
Let \( A \) and \( B \) be two abelian subalgebras of a transposed Novikov-Poisson algebra \( (P,\cdot,\circ) \). If \( P = A + B \), then \( P \) is metatrivial.
\end{thm}
\begin{proof}
By Proposition~\ref{metapro}, we only need to show that \( (c_1*_1c_2)*_2(c_3*_3c_4)= 0 \), where \( c_1,c_2,c_3,c_4 \in P \) and $*_1,*_2,*_3\in \{\cdot,\circ\}$.\par
For all $c_1,c_2,c_3,c_4\in P$, we have  $(c_1\cdot c_2)\cdot (c_3\cdot c_4)=0$ directly.
Let $c_i=a_i+b_i$, where $i=1,2,3,4$, $a_i\in A$ and $b_i\in B$. Then we obtain
\begin{eqnarray*}
(c_1\cdot c_2)\circ (c_3\cdot c_4)&=&(a_1\cdot b_2)\circ (a_3\cdot b_4)+(a_1\cdot b_2)\circ (b_3\cdot a_4)\\
&+&(b_1\cdot a_2)\circ (a_3\cdot b_4)+(b_1\cdot a_2)\circ (b_3\cdot a_4)\\
&=&(a_1\cdot (a_3\cdot b_4))\circ b_2+(a_1\cdot (b_3\cdot a_4))\circ b_2\\
&+&(b_1\cdot (a_3\cdot b_4))\circ a_2+(b_1\cdot  (b_3\cdot a_4))\circ a_2\\
&=&0.
\end{eqnarray*}
Similarly, we obtain $(c_1\circ c_2)\cdot c_3\cdot c_4=(c_1\cdot c_2)\cdot (c_3\circ c_4)=0$.\par
By Eqs.~\eqref{transNP1} and~\eqref{transNP2}, we obtain
\begin{eqnarray*}
a\circ (a'\cdot b')= 2b'\cdot(a\circ a')-(b'\cdot a)\circ a'=2 b'\cdot(a\circ a')-(a\cdot a')\circ b'=0
\end{eqnarray*}
and
\begin{eqnarray*}
a\cdot (a'\circ b')=\frac{1}{2}(a\cdot a')\circ b+\frac{1}{2}a'\circ (a\cdot b')=0+b'\cdot (a'\circ a)-\frac{1}{2}(b'\cdot a')\circ a=0,
\end{eqnarray*}
for all $a,a'\in A,b, b'\in B$. Similarly, we have $a\cdot (b'\circ a')=b\cdot (a'\circ b')=b\cdot (b'\circ a')=0$. Thus,  we can obtain $(c_1\circ c_2)\cdot (c_3\circ c_4)=(c_1\cdot c_2)\circ (c_3\circ c_4)=(c_1\circ c_2)\circ (c_3\cdot c_4)=0$ directly.
And $(c_1\circ c_2)\circ (c_3\circ c_4)=0$ follows from It\^{o}'s theorem for Novikov algebras~\cite{Nilpotent Novikov}.
\end{proof}

According to Theorem~\ref{ito}, if $A$ and $B$ are abelian, the whole transposed Novikov-Poisson algebra $P=A+B$ is 2-step solvable, and by Theorem~\ref{nil-soltrans}, it is right nilpotent. However, it may not be nilpotent. Finally, we present an example to show this.

\begin{ex}
Let $P$ be a 2-dimensional vector space over {\bf k} with a basis $\{a,b\}$. By~\cite{Nilpotent Novikov}, $(P, \circ)$ is a Novikov algebra with the non-zero products given by
\begin{eqnarray*}
&&a\circ b = b.
\end{eqnarray*}
Now we define the trivial associative products. It is easy to see that $P$ is a transposed Novikov-Poisson algebra and ${\bf k}a$ and ${\bf k}b$ are abelian subalgebras, and $(P,\cdot, \circ)$ is not nilpotent.
\end{ex}

\noindent {\bf Acknowledgments.}
This research is supported by
Natural Science Foundation of China (No. 12171129) and  Zhejiang
Provincial Natural Science Foundation of China (No. Z25A010006).

\smallskip

\noindent
{\bf Declaration of interests. } The authors have no conflicts of interest to disclose.

\smallskip

\noindent
{\bf Data availability. } No new data were created or analyzed in this study.

\end{document}